\documentclass[english,a4paper,12pt]{article}
\usepackage[utf8]{inputenc}
\usepackage{babel,amsmath,amsthm,amssymb,mathrsfs,graphicx}
\usepackage[sort&compress,numbers,square,sort]{natbib}

\makeatletter
\let\c@paragraph\relax
\makeatother
\newcounter{paragraph}[section]
\theoremstyle{plain}
\newtheorem*{theorem}{Theorem}
\newtheorem*{lemma}{Lemma}
\newtheorem*{proposition}{Proposition}
\theoremstyle{definition}
\newtheorem*{remark}{Remarks}

\renewcommand{\tilde}{\widetilde}
\renewcommand{\P}{\mathsf{P}}
\newcommand{\E}{\mathsf{E}}
\newcommand{\I}{\mathbf{I}}
\newcommand{\M}{{\mathfrak{M}_T}}
\newcommand{\F}{\mathscr{F}}

\title{Optimal stopping problems for a Brownian motion with a disorder on a
finite interval}
\author{A.\,N. Shiryaev\thanks{Steklov Mathematical Instutute, Moscow,
Russia; email: albertsh@mi.ras.ru.} \and 
M.\,V. Zhitlukhin\thanks{Steklov Mathematical Instutute, Moscow,
Russia, and The University of Manchester, UK; email: mikhailzh@mi.ras.ru.}}

\begin{document}
\maketitle

\begin{abstract}
We consider optimal stopping problems for a Brownian motion and a geometric
Brownian motion with a ``disorder'', assuming that the moment of a disorder
is uniformly distributed on a finite interval. Optimal stopping rules are
found as the first hitting times of some Markov process (the Shiryaev--Roberts
statistic) to time-dependent boundaries, which are characterized by certain
Volterra integral equations. 

The problems considered are related to mathematical finance and can be
applied in questions of choosing the optimal time to sell an asset with
changing trend.

\end{abstract}

\section{Introduction}
\label{formulation}
\paragraph{}
In this paper we consider problems of stopping a Brownian motion or a
geometric Brownian motion optimally on a finite time interval when it has a
\emph{disorder}, i.\,e. its drift coefficient changes at some unknown moment of
time from a positive value to a negative one. We look for the stopping
time that maximizes the expected value of the stopped process.

Let $B=(B_t)_{t\ge0}$ be a standard Brownian motion defined on a probability
space $(\Omega,\F,\P)$. Suppose we sequentially observe
the process $X=(X_t)_{t\ge0}$,
\[
X_t = \mu_1t + (\mu_2-\mu_1)(t-\theta)^+ + \sigma B_t,
\]
or, equivalently in stochastic differentials,
\[
d X_t = \bigl[\mu_1 \I(t<\theta) + \mu_2 \I(t\ge \theta)\bigr] dt + \sigma d
B_t, \qquad X_0=0,
\]
where $\mu_1>0>\mu_2$, $\sigma>0$ are known numbers and $\theta$ is an
unknown \emph{time of a disorder} -- a moment when the drift coefficient of
$X$ changes from value $\mu_1$ to value $\mu_2$. The process $X$ is
called a (linear) \emph{Brownian motion with a disorder}.

Adopting the Bayesian approach, we let $\theta$ be a random variable defined
on $(\Omega,\F,\P)$ and independent of $B$. In this paper, in view of
applications (see below), we assume that $\theta$ is \emph{uniformly
distributed} on a finite interval $[0,T]$, possibly, with mass at $t=0,T$,
i.\,e. the distribution function $G(t) = \P(\theta\le t)$ is given by
\[
G(t) = G(0) + \rho t\text{ for } 0\le t< T, \qquad G(T) = 1,
\]
where $G(0)\in[0,1)$ is the probability that the disorder presents from the
beginning, and $0< \rho \le (1-G(0))/T$ is the density of $\theta$. The
probability $\P(\theta=1)=1-G(T-)$ may be strictly positive.

Let $\M$ denote the class of all stopping times $\tau\le T$ of the process
$X$. We consider the following two optimal stopping problems for $X$ and the
exponent of $X$ (a \emph{geometric Brownian motion with a disorder}):
\begin{equation}
V^{(l)} = \sup_{\tau\in\M} \E X_\tau,
\qquad V^{(g)} = \sup_{\tau\in\M} \E \exp(X_\tau - \sigma^2\tau/2).\label{1}
\end{equation}
The problems consist in finding the values $V^{(l)}$, $V^{(g)}$ and finding
the stopping times $\tau_*^{(l)}$, $\tau_*^{(g)}$ at which the suprema are
attained (if such stopping times exist). The superscript ${(l)}$ stands for
the problem for a \emph{linear} Brownian motion, while $(g)$ stands for a
\emph{geometric} Brownian motion.

Observe that, roughly speaking, the processes $X_t$ and
$\exp(X_t-\sigma^2t/2)$ increase on average ``up to time $\theta$'' and
decrease on average ``after time $\theta$''. But since $\theta$ is not a
stopping time, we cannot simply take $\tau_*=\theta$ and need to stop by
detecting the disorder based on sequential observation of $X$.

We provide solutions to problems \eqref{1} using the results obtained in
the recent paper \cite{ZS12-e}. The central idea is based on a reduction to
a Markovian optimal stopping problem using a change of measure. This
approach has already been applied in the literature, but we were able to
generalize it to the case of a finite time interval.

\paragraph{}
For economic applications, the problems considered are related to the
question when to quit a financial ``bubble''. By a bubble we mean growth
of an asset price based mainly on the expectation of higher future price;
eventually a bubble bursts and the price starts to decline.

Suppose that an asset price is modelled by a geometric Brownian
motion with a disorder $S=(S_t)_{t\ge0}$:
\[
S_t = \exp(X_t - \sigma^2t/2),
\]
or, equivalently,
\[
d S_t = \bigl[\mu_1 \I(t<\theta) + \mu_2 \I(t\ge\theta)\bigr] S_t dt +
\sigma S_t d B_t, \qquad S_0 = 1.
\]
Thus the price initially has a positive trend, which changes to a negative
one at an unknown time $\theta$.

Let $t=0$ correspond to the ``current'' moment of time, when one is in a
long position on an asset with positive trend. Usually, it is possible to
predict that the trend will become negative by some (maybe, distant) time $T$ in
the future. Then one is interested in the question when it is optimal to
sell the asset maximizing the gain.

If nothing is known about the actual distribution of $\theta$, it is natural
to assume that $\theta$ is uniformly distributed on $[0,T]$ (since the
uniform distribution has the maximum entropy on a finite
interval). Interpreting the quantity $\E S_\tau$ as the average gain
achieved by selling the asset at time $\tau$, problem \eqref{1} for a
geometric Brownian motion seeks for the optimal time to sell the asset. The
problem for a linear Brownian motion can be thought of as a problem of
finding the optimal time to sell the asset provided that a trader has the
logarithmic utility function, i.\,e. maximizes $\E\log(S_\tau)$, which is
equivalent to maximizing $\E X_\tau$ with $\mu_1' = \mu_1 -
\sigma^2/2$, $\mu_2' = \mu_2 - \sigma^2/2$.

An interesting result that follows from the solution of problems \eqref{1}
is the qualitative difference between \emph{risk-neutral} traders who
maximize $\E S_\tau$ and \emph{risk-averse} traders who maximize
$\E\log(S_\tau)$: if $\P(\theta<1)=1$ (i.\,e. the distribution of $\theta$
has no mass at $t=T$), then a risk-neutral trader will sell the asset
strictly before time $T$ with probability one ($\P(\tau_*^{(g)}<T)=1$),
while a risk-averse trader will wait until the end of the time interval with
positive probability ($\P(\tau_*^{(l)}=T)>0$); see the Theorem and 
Remark~1 in Section~\ref{results}.

\paragraph{}
Problems \eqref{1} was considered in the papers \cite{BL97,NS09,EL13},
assuming that the moment of disorder $\theta$ is exponentially
distributed. In \cite{BL97}, the problem for a geometric Brownian motion was
solved. It was shown that if $\mu_1,\mu_2,\sigma$ satisfy some relation, the
optimal stopping time is the first hitting time of the posterior probability
process $\pi_t = \P(\theta\le t \mid \F_t^X)$, where
$\F_t^X=\sigma(X_s;s\le t)$, to some level. In \cite{EL13} this
result was extended to all values of $\mu_1,\mu_2,\sigma$ and the optimal
stopping level was found. In \cite{NS09}, the problem for a linear Brownian
motion was considered on a finite interval, i.\,e. assuming that one should
choose $\tau$ not exceeding some time horizon $T$ (but the disorder may
happen after $T$). It turned out that the problem is equivalent to the
original Bayesian setting of the disorder detection problem when one seeks for a
stopping time minimizing the average detection delay and the probability of
a false alarm (see e.\,g. \cite{S63a-e,S76-e,GP06}). The paper \cite{NS09}
also briefly discusses the optimal stopping problem for a geometric Brownian
motion on a finite interval, but does not provide an explicit solution.

\section{The main result}
\label{results}
\paragraph{}
Let $\mu=(\mu_1-\mu_2)/\sigma$ denote the \emph{signal-to-noise ratio}.  For
convenience of notation, introduce the process $\tilde X = (\tilde
X_t)_{t\ge0}$, $\tilde X_t = (X_t - \mu_1 t)/\sigma$, which is a Brownian
motion with the unit diffusion coefficient and the drift coefficient
changing at time $\theta$ from value $0$ to value $(-\mu)$.

Introduce the \emph{Shiryaev--Roberts statistic}\footnote{In a general case, the
Shiryaev--Roberts statistic is given by $\psi_t = (d\P_t^0/d\P_t^\infty) \int_0^t
(d\P^\infty_s/d\P^0_s) dG(s)$, where $\P^s_t=\P^s\mid\F_t^X$ for the measures $\P^0$
or $\P^\infty$, corresponding to that the disorder happens at time $t=0$ or
does not happen at all (for details, see~\cite{ZS12-e} and the proof of
the Lemma in Section \ref{proofs}).} $\psi=(\psi_t)_{t\ge0}$:
\begin{equation}\label{psi-def}
\psi_t = \exp\bigl(-\mu \tilde X_t - \mu^2 t/2\bigr) \left(\psi_0 + \rho
\int_0^t \exp\bigl(\mu \tilde X_s + \mu^2 s/2 \bigr) ds \right)
\end{equation}
with $\psi_0 = G(0)$. Applying the It\^o formula it is easy to see that
$\psi$ satisfies the stochastic differential equation
\begin{equation}\label{psi-sde}
d\psi_t = \rho dt - \mu \psi_t d \tilde X_t.
\end{equation}

On the measurable space $(\Omega,\F^X_T)$, $\F^X_T=\sigma(X_t;t\le T)$,
define the probability measures $\P^{(l)}$ and $\P^{(g)}$ such that $\tilde
X_t$ is a standard Brownian motion under $\P^{(l)}$ and $(\tilde X_t-\sigma
t)$ is a standard Brownian motion under $\P^{(g)}$. These measures will be
used to solve the problems $V^{(l)}$ and $V^{(g)}$ respectively. It is
well-known (see, e.\,g., \cite[Ch.~7]{LS00}) that $\P^{(l)}$ and $\P^{(g)}$
are equivalent on the space $(\Omega,\F^X_T)$.

For any $x\ge 0$, by $\E^{(l)}_x[\,\cdot\,]$ and $\E^{(g)}_x[\,\cdot\,]$ we
denote the mathematical expectations of functionals of the process
$(\psi_{t})_{t\ge0}$ defined by \eqref{psi-def}--\eqref{psi-sde} with the
initial condition $\psi_0=x$, when $\tilde X$ is respectively a standard
Brownian motion or a Brownian motion with drift $\sigma$. For brevity,
instead of $\E^{(l)}_{G(0)}[\,\cdot\,]$ and $\E^{(g)}_{G(0)}[\,\cdot\,]$
we simply write $\E^{(l)}[\,\cdot\,]$ and $\E^{(g)}[\,\cdot\,]$.

The main result of the paper is the following theorem.
\begin{theorem}
The optimal stopping times in the problems $V^{(l)}$ and $V^{(g)}$ are given
respectively by
\begin{align*}
&\tau^{(l)}_* = \inf\{t\ge 0: \psi_t \ge a^{(l)}(t)\} \wedge T,\\
&\tau^{(g)}_* = \inf\{t\ge 0: \psi_t \ge a^{(g)}(t)\} \wedge T,
\end{align*}
where $a^{(l)}(t)$, $a^{(g)}(t)$ are non-increasing functions on $[0,T]$
being the unique solutions of the integral equations ($t\in[0,T]$)
\begin{align}
&\int_0^{T-t} \E_{a^{(l)}(t)}^{(l)}
\bigl[\bigl(\mu_1-(\mu_1-\mu_2)\psi_s\bigr)\I\{\psi_s < a^{(l)}(t+s)\}\bigr] ds = 0, \label{al-eq}\\
&\int_0^{T-t} \E_{a^{(g)}(t)}^{(g)} \bigl[e^{\mu_1s}\bigl(\mu_1
(1-G(t+s)) - |\mu_2| \psi_s \bigr) \I\{\psi_s < a^{(g)}(t+s)\} \bigr] ds = 0 \label{ag-eq}
\end{align}
in the class of continuous bounded functions on $[0,T)$ satisfying the conditions
\begin{align}
&a^{(l)}(t) \ge \frac{\mu_1}{\mu_1-\mu_2} \text{ for }t\in[0,T),& 
&a^{(l)}(T) = \frac{\mu_1}{\mu_1-\mu_2}, \label{al-cond}\\
&a^{(g)}(t) \ge \frac{\mu_1}{|\mu_2|} (1-G(t)) \text{ for }t\in[0,T),&
&a^{(g)}(T) = \frac{\mu_1}{|\mu_2|} (1-G(T-)). \label{ag-cond}
\end{align}
The values $V^{(l)}$ and $V^{(g)}$ can be found by the formulas
\begin{align}
&V^{(l)} = \int_0^T \E^{(l)} \bigl[\bigl(\mu_1-(\mu_1-\mu_2)\psi_s\bigr)\I\{\psi_s < a^{(l)}(s)\} \bigr] ds, \label{l-profit}\\
&V^{(g)} = 1 + \int_0^T\E^{(g)} \bigl[e^{\mu_1s}\bigl(\mu_1
(1-G(s)) - |\mu_2| \psi_s \bigr) \I\{\psi_s < a^{(g)}(s)\} \bigr] ds. \label{g-profit}
\end{align}
\end{theorem}

\begin{remark}
\textbf{1.} Regarding the difference between a risk-averse trader and a
risk-neutral trader mentioned in Section \ref{formulation}, observe that if
$G(T-)=1$, then $\P(\tau_*^{(g)}<T)=1$ because $a^{(g)}(T)=0$, while
$\P(\tau_*^{(l)}<T)<1$ because $a^{(l)}(T)>0$ (in the latter case the
process $\psi$ stays below $a^{(l)}(t)$ on the whole interval $[0,T]$ with
positive probability).

\textbf{2.} In the above mentioned papers \cite{BL97,NS09,EL13}, solutions to problems
\eqref{1} when $\theta$ is exponentially distributed were given in terms of
the posterior probability process $\pi_t = \P(\theta\le t \mid
\F_t^X)$. Using the Bayes formula, one can check that the processes
$\pi$ and $\psi$ are connected by the formula $\psi_t =
\pi_t(1-G(t))/(1-\pi_t)$ (see~\cite{ZS12-e}). Consequently, it is easy to
reformulate the Theorem in a such way that $\tau_*^{(l)}$ and $\tau_*^{(g)}$
are the first moments of time when the process $\pi$ crosses time-dependent
levels. We prefer to work with the process $\psi$ because it has a somewhat
simpler form than $\pi$, when $\theta$ is uniformly distributed.
\end{remark}

\paragraph{}
Equations \eqref{al-eq}, \eqref{ag-eq} can be solved numerically by
\emph{backward induction}: we fix a partition $0\le t_0<t_1<\ldots<t_n=T$ of
 $[0,T]$ and sequentially find the values
$a(t_n),\;a(t_{n-1}),\;\ldots,\allowbreak a(t_0)$. The value $a(T)$ can be
found from condition \eqref{al-cond} or \eqref{ag-cond} respectively. Having
found the values $a(t_k), a(t_{k+1}), \ldots,a(t_n)$ and numerically
computing integral \eqref{al-eq} or \eqref{ag-eq} for $t=t_{k-1}$ through
the values of the integrand at points $t_{k-1},t_k,\ldots,\allowbreak t_n$,
we obtain the equation, from which the value $a(t_{k-1})$ can be
found. Repeating this procedure, we find the value of $a(t)$ at every point
of the partition.

To compute the mathematical expectations in \eqref{al-eq}, \eqref{ag-eq},
\eqref{l-profit}, \eqref{g-profit}, one can use the Monte--Carlo method or
use the explicit formula for the transitional density of $\psi$ (see
e.\,g. \cite{ZS12-e}, where it was obtained from the joint law of an
exponent of a Brownian motion and its integral that can be found in
\cite{MY05}).

As an example of a numerical solution of equations \eqref{al-eq},
\eqref{ag-eq}, on Figure~\ref{fig-1} we present the optimal stopping boundaries
for the case $T=1$, $\mu_1=1$, $\mu_2=-1$, $\sigma=1$, $G(t) = t$ for
$t\le 1$.

\begin{figure}[h]
\centering
\includegraphics[width=6.6cm]{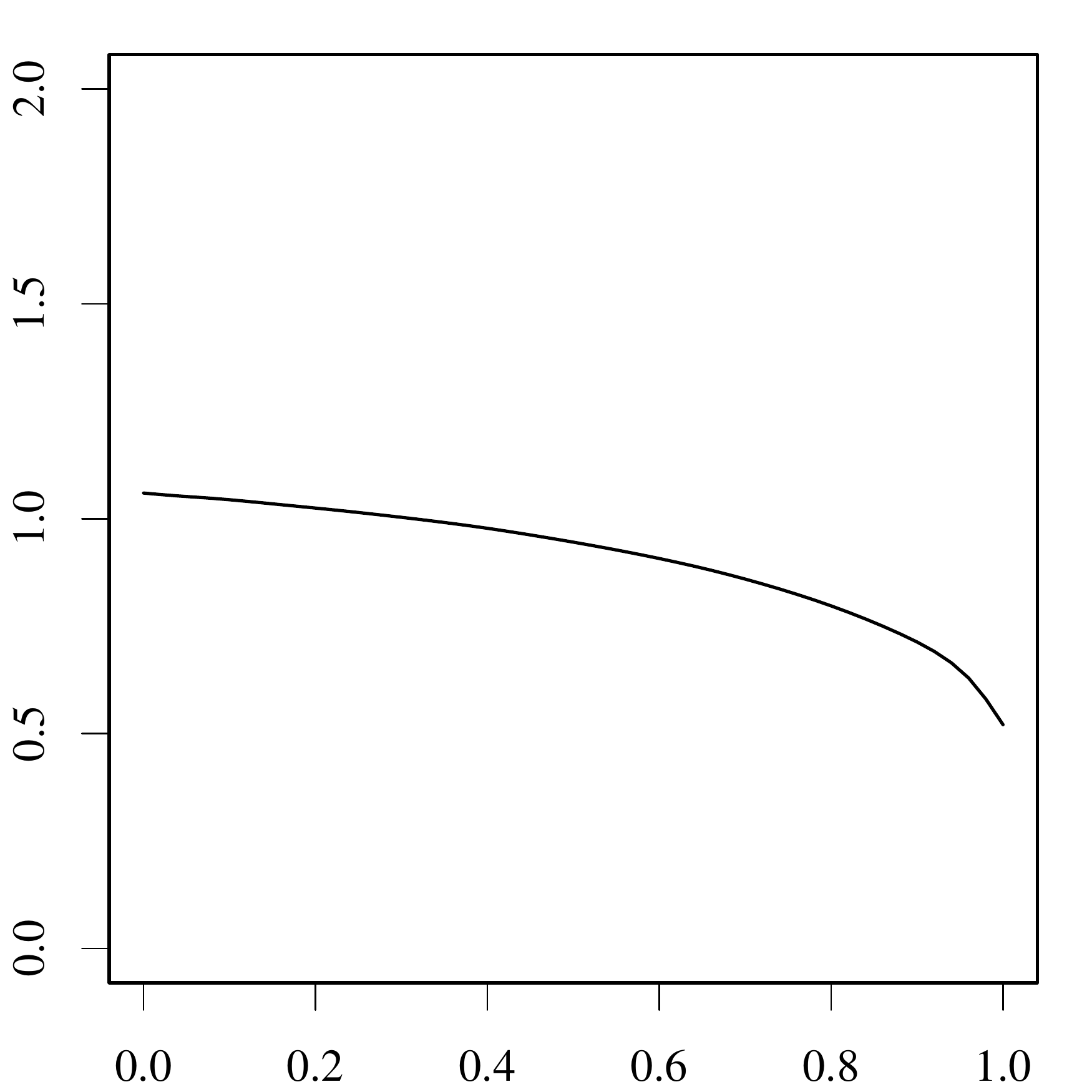}
\hfill
\includegraphics[width=6.6cm]{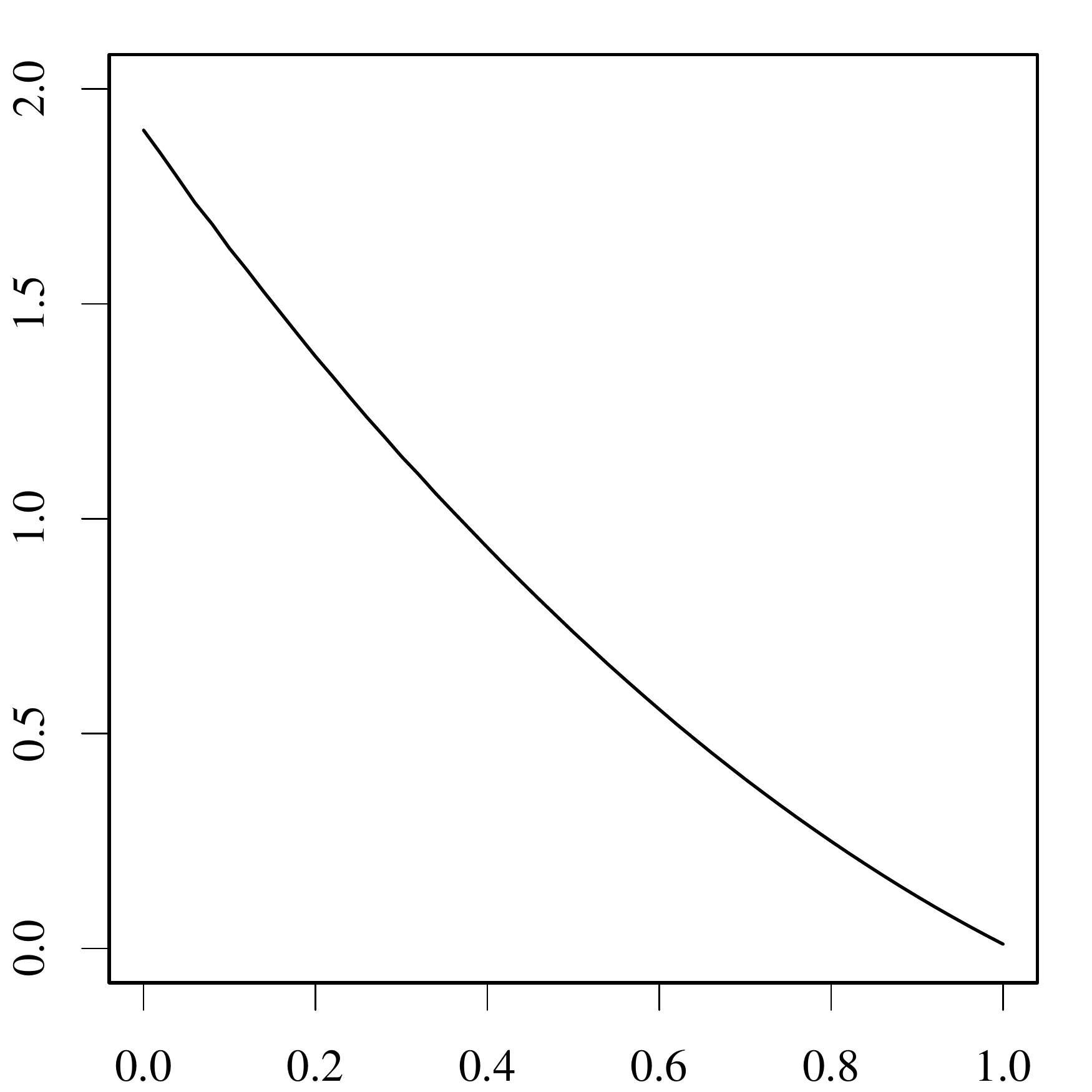}
\caption{The optimal stopping boundaries in the problems $V^{(l)}$ (left) and
$V^{(g)}$ (right) in the case $T=1$, $\mu_1=1$, $\mu_2=-1$, $\sigma=1$,
$G(t) = t$ for $t\le 1$.}
\label{fig-1}
\end{figure}

\section{Proof of the theorem}
\label{proofs}
\paragraph{}
To prove the theorem, we first reduce problems $V^{(l)}$ and $V^{(g)}$ to
optimal stopping problems for the process $\psi$, and then apply the
Proposition from the Appendix, which was proved in \cite{ZS12-e}. The method
we use is based on the ideas of \cite{FS06, BL97}.

\begin{lemma}
The following formulae hold:
\begin{align*}
&V^{(l)} = \sup_{\tau\in\M} \E^{(l)} \int_0^\tau \bigl(\mu_1 - (\mu_1-\mu_2)
\psi_s\bigr) ds,\\ 
&V^{(g)}  = 1 + \sup_{\tau\in\M} \E^{(g)} \int_0^\tau e^{\mu_1 s}
\bigl[\mu_1(1-G(s)) - |\mu_2| \psi_s \bigr] ds.
\end{align*}

The supremum in each formula is attained at the stopping time which is
optimal in the corresponding problem \eqref{1}.
\end{lemma}

\begin{proof}
It is sufficient to show that for any stopping time $\tau\in\M$
\begin{align}
&\E X_\tau = \E^{(l)} \int_0^\tau \bigl(\mu_1 - (\mu_1-\mu_2) \psi_s\bigr)
ds, \label{eq:4}\\
&\E \exp(X_\tau - \sigma^2\tau/2) = 1 + \E^{(g)} \int_0^\tau e^{\mu_1 s}
\bigl[\mu_1(1-G(s)) - |\mu_2| \psi_s \bigr] ds.\label{eq:5}
\end{align}

On the measurable space $(\Omega,\F^X_T)$ define the family of
probability measures $(\P^u)_{0\le u\le T}$ such that under $\P^u$ the
disorder occurs at the fixed time $u$, i.\,e. for each $0\le u \le T$ the
process $ X$ can be represented as $ X_t=\mu_1t + (\mu_2-\mu_1)(t-u)^+ +
\sigma B_t^{(u)}$, where $B^{(u)}$ is a standard Brownian motion under
$\P^u$. By $\E^u[\,\cdot\,]$ we denote the mathematical expectation with
respect to $\P^u$ and by $\P_t$, $\P^u_t$, $\P_t^{(l)}$, $\P_t^{(g)}$ we
denote the restrictions of the corresponding measures to the
$\sigma$-algebra $\F_t^X=\sigma(X_s;s\le t)$, $0\le t \le T$.

Let us prove \eqref{eq:4}. Since the Brownian motion $B$ is a zero-mean
martingale, we have
\begin{equation}\label{eq:6}
\E X_\tau = \E\bigl[ \mu_1\tau +
(\mu_2-\mu_1)(\tau-\theta)^+\bigr].
\end{equation}
Consider the second term in the sum:
\begin{equation}
\E (\tau-\theta)^+ = \int_0^T \E^u[(\tau-u)\I(\tau>u)] d G(u).\label{eq:1}
\end{equation}
Observe that for any $0\le u\le T$ the following relation is valid:
\begin{multline}
\E^u [(\tau-u)\I(\tau>u)] =
\int_u^T \E^u\I(s\le\tau) ds  \\=
\int_u^T \E^{(l)} \bigl[ \exp\bigl(-\mu(\tilde X_s-\tilde X_u) -
\mu^2(s-u)/2\bigr) \I(s\le\tau) 
\bigr] ds,\label{eq:2}
\end{multline}
where we use that $\I(s\le\tau)$ is an $\F_s^X$-measurable random
variable and
\[
\frac{d\P^u_s}{d\P^{(l)}_s} = \exp\bigl(-\mu(\tilde X_s-\tilde X_u) -
\mu^2(s-u)/2\bigr)
\]
(the explicit formula for the density of the measure generated by one It\^o
process with respect to the measure generated by another It\^o process can
be found in e.\,g. \cite{LS00}).  From \eqref{eq:1}--\eqref{eq:2}, changing
the order of integration we find
\[
\begin{aligned}
\E (\tau-\theta)^+ &= \E^{(l)} \int_0^\tau \int_0^s \exp\bigl(-\mu(\tilde X_s-\tilde
X_u) - \mu^2(s-u)/2\bigr) dG(u) ds \\ &= \E^{(l)} \int_0^\tau \psi_s ds.
\end{aligned}
\]
Then using \eqref{eq:6}, we obtain representation \eqref{eq:4}. 

Let us prove \eqref{eq:5}. We have
\[
\frac{d\P^s_t}{d \P_t^{(g)}} = \begin{cases} \exp\bigl(-\sigma \tilde X_t
+\sigma^2t/2\bigr)\cdot \exp\bigl(-\mu(\tilde X_t-\tilde X_s\bigr) -
\mu^2(t-s)/2), &s\le t,\\
\exp\bigl(-\sigma \tilde X_t +\sigma^2t/2\bigr), &s> t,
\end{cases}
\]
which implies
\[
\begin{split}
\frac{d \P_t}{d\P^{(g)}_t} &= \int_0^T \frac{d\P^s_t}{d \P^{(g)}_t} d
G(s) \\ &=
\begin{aligned}[t]\exp\bigl(&-\sigma \tilde X_t + \sigma^2t/2\bigr) \\
&\times \left( \int_0^t \exp\bigl(-\mu(\tilde X_t-\tilde X_s) -
\mu^2(t-s)/2\bigr)\, dG(s) + 1-G(t) \right)
\end{aligned}
\\ &=
\exp\bigl(-\sigma \tilde X_t +\sigma^2t/2\bigr) \bigl( \psi_t + 1 - G(t)\bigr).
\end{split}
\]
Consequently,
\[
\begin{split}
\E \exp(X_\tau - \sigma^2\tau/2) &= \E^{(g)} \left(\frac{d\P_\tau}{d\P_\tau^{(g)}}
\exp(X_\tau-\sigma^2\tau/2) \right) \\ &=
\E^{(g)} \bigl[ e^{\mu_1 \tau}(\psi_\tau+1-G(\tau)) \bigr].
\end{split}
\]
Applying the It\^o formula, we get
\begin{multline*}
e^{\mu_1 \tau}(\psi_\tau+1-G(\tau)) \\= 1 + \int_0^\tau e^{\mu_1 s} \left[\mu_2 \psi_s
+ \mu_1(1-G(s))\right] ds + \mu \int _0^\tau \psi_s d (\tilde X_s-\sigma s).
\end{multline*}
Taking the mathematical expectation $\E^{(g)}[\,\cdot\,]$ of the both sides
and using that $(\tilde X_s-\sigma s)$ is a standard Brownian motion under
$\P^{(g)}$ and, hence, the expectation of the stochastic integral equals
zero, we obtain \eqref{eq:5}.
\end{proof}

\paragraph{}
Now the proof of the Theorem follows from the Proposition in the
Appendix -- for the problem $V^{(l)}$ we use that the process $\psi$
satisfies equation \eqref{psi-sde} with $\tilde X$ being a Brownian motion
under the measure $\P^{(l)}$, and for the problem $V^{(g)}$ we use that
$\psi$ satisfies the equation
\begin{equation}
d \psi_t = \bigl(\rho - (\mu_1-\mu_2) \psi_t\bigr) dt - \mu \psi_t d(\tilde
X_t - \sigma t),\label{eq:3}
\end{equation}
where $(\tilde X_t - \sigma t)$ is a Brownian motion under $\P^{(g)}$.

\section{Acknowledgements}
The authors are grateful to Prof.~H.\,R.~Lerche and A.\,A.~Levinskiy for
valuable remarks. The work is supported by Laboratory for Structural Methods
of Data Analysis in Predictive Modeling, MIPT, RF government grant,
ag. 11.G34.31.0073. The work of M.\,V.~Zhitlukhin is also supported by The
Russian Foundation for Basic Research, grant 12-01-31449-mol\_a.

\section*{Appendix}
Let $B=(B_t)_{t\ge0}$ be a standard Brownian motion defined on a filtered
probability space $(\Omega,\F,(\F_t)_{t\ge0},\P)$ and
$\psi=(\psi_t)_{t\ge 0}$ be a stochastic process satisfying the stochastic
differential equation (cf. \eqref{psi-sde},~\eqref{eq:3})
\begin{equation}\label{eq:28}
d \psi_t = (\rho + b \psi_t) dt - \mu \psi_t dB_t,
\end{equation}
where $b,\mu\in\mathbb{R}$ and $\rho> 0$.

Consider the optimal stopping problem consisting in finding the quantity
\begin{equation}\label{eq:27}
V = \sup_{\tau \le T} \E \int_0^\tau e^{\lambda s} (f(s)-\psi_s) ds, 
\end{equation}
where $\lambda\in \mathbb{R}$, and $f(s)$ is a non-increasing bounded
function which is continuous and strictly positive on $[0,T)$. The supremum
in \eqref{eq:27} is taken over all stopping times $\tau$ of the filtration
$(\F_t)_{t\ge0}$ satisfying the condition $\tau\le T$ a.s.

For arbitrary $x\ge 0$, let $\E_x[\,\cdot\,]$ denote the
mathematical expectation of functionals of the process
$(\psi_{t})_{t\ge0}$, satisfying \eqref{eq:28} with
$\psi_0 = x$.

The following result was proved in \cite{ZS12-e}.

\begin{proposition}
The optimal stopping time in problem \eqref{eq:27} is given by 
\[
\tau_* = \inf\{t\ge 0: \psi_t \ge a(t)\} \wedge T,
\]
where $a(t)$ is a non-increasing function on $[0,T]$ being the
unique solution of the equation ($t\in[0,T]$)
\[
\int_0^{T-t} \E_{a(t)} \bigl[e^{\lambda s}(f(t+s)-\psi_s) \I\{\psi_s < a(t+s)\}
\bigr] ds = 0
\]
in the class of continuous bounded functions on $[0,T)$ satisfying the conditions
\[
a(t) \ge f(t) \text{ for }t\in[0,T), \qquad
a(T) = f(T-).
\]
The quantity $V$ can be found by the formula
\[
V= \int_0^T \E_{\psi_0}\bigl[e^{\lambda s}(\psi_s-f(s)) \I\{\psi_s<a(s)\}\bigr] ds.   
\]
\end{proposition}

\setlength{\bibsep}{0.65em}

\end{document}